\documentclass{amsart}
\usepackage{amsmath}
\usepackage{amsfonts}

\newtheorem{theorem}{Theorem}
\theoremstyle{plain}

\newtheorem{corollary}{Corollary}

\newtheorem{definition}{Definition}

\newtheorem{lemma}{Lemma}

\numberwithin{equation}{section}
\input{tcilatex}

\begin{document}
\title[New general integral inequalities]{New general integral inequalities
for some GA-convex and quasi-geometrically convex functions via fractional
integrals}
\author{\.{I}mdat \.{I}\c{s}can}
\address{Department of Mathematics, Faculty of Sciences and Arts, Giresun
University, Giresun, Turkey}
\email{imdat.iscan@giresun.edu.tr}

\begin{abstract}
In this paper, the author introduces the concept of the quasi-geometrically
convex and defines a new identity for fractional integrals. By using of this
identity, author obtains new estimates on generalization of Hadamard,
Ostrowski and Simpson type inequalities for GA-$s$-convex,
quasi-geometrically convex and $(s,m)$-GA-convex functions via Riemann
Liouville fractional integral.
\end{abstract}

\maketitle

\section{Introduction}

Let real function $f$ be defined on some nonempty interval $I$ of real line $%
\mathbb{R}
$. The function $f$ is said to be convex on $I$ if inequality%
\begin{equation}
f(tx+(1-t)y)\leq tf(x)+(1-t)f(y)  \label{1-0}
\end{equation}%
holds for all $x,y\in I$ and $t\in \left[ 0,1\right] .$

Following inequalities are well known in the literature as Hermite-Hadamard
inequality, Ostrowski inequality and Simpson inequality respectively:

\begin{theorem}
Let $f:I\subseteq \mathbb{R\rightarrow R}$ be a convex function defined on
the interval $I$ of real numbers and $a,b\in I$ with $a<b$. The following
double inequality holds%
\begin{equation}
f\left( \frac{a+b}{2}\right) \leq \frac{1}{b-a}\dint\limits_{a}^{b}f(x)dx%
\leq \frac{f(a)+f(b)}{2}\text{.}  \label{1-1}
\end{equation}
\end{theorem}

\begin{theorem}
Let $f:I\subseteq \mathbb{R\rightarrow R}$ be a mapping differentiable in $%
I^{\circ },$ the interior of I, and let $a,b\in I^{\circ }$ with $a<b.$ If $%
\left\vert f^{\prime }(x)\right\vert \leq M,$ $x\in \left[ a,b\right] ,$
then we the following inequality holds%
\begin{equation*}
\left\vert f(x)-\frac{1}{b-a}\dint\limits_{a}^{b}f(t)dt\right\vert \leq 
\frac{M}{b-a}\left[ \frac{\left( x-a\right) ^{2}+\left( b-x\right) ^{2}}{2}%
\right]
\end{equation*}%
for all $x\in \left[ a,b\right] .$ The constant $\frac{1}{4}$ is the best
possible in the sense that it cannot be replaced by a smaller one.
\end{theorem}

\begin{theorem}
Let $f:\left[ a,b\right] \mathbb{\rightarrow R}$ be a four times
continuously differentiable mapping on $\left( a,b\right) $ and $\left\Vert
f^{(4)}\right\Vert _{\infty }=\underset{x\in \left( a,b\right) }{\sup }%
\left\vert f^{(4)}(x)\right\vert <\infty .$ Then the following inequality
holds:%
\begin{equation*}
\left\vert \frac{1}{3}\left[ \frac{f(a)+f(b)}{2}+2f\left( \frac{a+b}{2}%
\right) \right] -\frac{1}{b-a}\dint\limits_{a}^{b}f(x)dx\right\vert \leq 
\frac{1}{2880}\left\Vert f^{(4)}\right\Vert _{\infty }\left( b-a\right) ^{4}.
\end{equation*}
\end{theorem}

The following de nitions are well known in the literature.

\begin{definition}[\protect\cite{N00,N03}]
A function $f:I\subseteq \left( 0,\infty \right) \rightarrow 
\mathbb{R}
$ is said to be GA-convex (geometric-arithmatically convex) if%
\begin{equation*}
f(x^{t}y^{1-t})\leq tf(x)+\left( 1-t\right) f(y)
\end{equation*}%
for all $x,y\in I$ and $t\in \left[ 0,1\right] $.
\end{definition}

\begin{definition}[\protect\cite{SYQ13}]
For $s\in \left( 0,1\right] ,$ a function $f:I\subseteq \left( 0,\infty
\right) \rightarrow 
\mathbb{R}
$ is said to be GA-s-convex (geometric-arithmatically s-convex) if%
\begin{equation*}
f(x^{t}y^{1-t})\leq t^{s}f(x)+\left( 1-t\right) ^{s}f(y)
\end{equation*}%
for all $x,y\in I$ and $t\in \left[ 0,1\right] $.
\end{definition}

It can be easily seen that for $s=1$, GA-s-convexity reduces to GA-convexity.

\begin{definition}[\protect\cite{JZQ13}]
Let $f:(0,b]\rightarrow 
\mathbb{R}
,b>0,$ and $\left( s,m\right) \in \left( 0,1\right] ^{2}$. If 
\begin{equation*}
f(x^{t}y^{m\left( 1-t\right) })\leq t^{s}f(x)+m\left( 1-t^{s}\right) f(y)
\end{equation*}%
for all $x,y\in (0,b]$ and $t\in \lbrack 0,1]$, then $f$ is said to be a $%
\left( s,m\right) $-GA-convex function .
\end{definition}

\begin{definition}[\protect\cite{N00,N03}]
A function $f:I\subseteq \left( 0,\infty \right) \rightarrow \left( 0,\infty
\right) $ is said to be GG-convex (called in \cite{ZJQ12} geometrically
convex function) if%
\begin{equation*}
f(x^{t}y^{1-t})\leq f(x)^{t}f(y)^{\left( 1-t\right) }
\end{equation*}%
for all $x,y\in I$ and $t\in \left[ 0,1\right] $.
\end{definition}

We give some necessary definitions and mathematical preliminaries of
fractional calculus theory which are used throughout this paper.

\begin{definition}
Let $f\in L\left[ a,b\right] $. The Riemann-Liouville integrals $%
J_{a^{+}}^{\alpha }f$ and $J_{b^{-}}^{\alpha }f$ of oder $\alpha >0$ with $%
a\geq 0$ are defined by

\begin{equation*}
J_{a^{+}}^{\alpha }f(x)=\frac{1}{\Gamma (\alpha )}\dint\limits_{a}^{x}\left(
x-t\right) ^{\alpha -1}f(t)dt,\ x>a
\end{equation*}

and

\begin{equation*}
J_{b^{-}}^{\alpha }f(x)=\frac{1}{\Gamma (\alpha )}\dint\limits_{x}^{b}\left(
t-x\right) ^{\alpha -1}f(t)dt,\ x<b
\end{equation*}%
respectively, where $\Gamma (\alpha )$ is the Gamma function defined by $%
\Gamma (\alpha )=$ $\dint\limits_{0}^{\infty }e^{-t}t^{\alpha -1}dt$ and $%
J_{a^{+}}^{0}f(x)=J_{b^{-}}^{0}f(x)=f(x).$
\end{definition}

In the case of $\alpha =1$, the fractional integral reduces to the classical
integral. Properties concerning this operator can be found \cite%
{GM97,MR93,P99}.

In recent years, many athors have studied errors estimations for
Hermite-Hadamard, Ostrowski and Simpson inequalities; for refinements,
counterparts, generalization see \cite%
{ADDC10,AKO11,D10,I13,I13b,I13c,OAK12,P11,SA11,S12,SO12,SOS12,SSO10,SSYB11}.

In this paper, new identity for fractional integrals have been defined. By
using of this identity, we obtained a generalization of Hadamard, Ostrowski
and Simpson type inequalities for GA-$s$-convex, quasi-geometrically convex, 
$(s,m)$-convex functions via Riemann Liouville fractional integral.

\ 

\section{Generalized integral inequalities for some GA-convex functions via
fractional integrals}

Let $f:I\subseteq \left( 0,\infty \right) \rightarrow 
\mathbb{R}
$ be a differentiable function on $I^{\circ }$, the interior of $I$,
throughout this section we will take%
\begin{eqnarray*}
I_{f}\left( x,\lambda ,\alpha ,a,b\right) &=&\left( 1-\lambda \right) \left[
\ln ^{\alpha }\frac{x}{a}+\ln ^{\alpha }\frac{b}{x}\right] f(x)+\lambda %
\left[ f(a)\ln ^{\alpha }\frac{x}{a}+f(b)\ln ^{\alpha }\frac{b}{x}\right] \\
&&-\Gamma \left( \alpha +1\right) \left[ J_{\left( \ln x\right)
^{-}}^{\alpha }\left( f\circ \exp \right) (\ln a)+J_{\left( \ln x\right)
^{+}}^{\alpha }\left( f\circ \exp \right) (\ln b)\right]
\end{eqnarray*}%
where $a,b\in I$ with $a<b$, $\ x\in \lbrack a,b]$ , $\lambda \in \left[ 0,1%
\right] $, $\alpha >0$ and $\Gamma $ is Euler Gamma function. In order to
prove our main results we need the following identity.

\begin{theorem}
Let $f:I\subseteq \left( 0,\infty \right) \rightarrow 
\mathbb{R}
$ be a function such that $f\in L[a,b]$, where $a,b\in I$ with $a<b$. If $f$
is a GA-convex function on $[a,b]$, then the following inequalities for
fractional integrals hold:%
\begin{equation}
f\left( \sqrt{ab}\right) \leq \frac{\Gamma (\alpha +1)}{2\left( \ln \frac{b}{%
a}\right) ^{\alpha }}\left\{ J_{\left( \ln a\right) ^{+}}^{\alpha }\left(
f\circ \exp \right) (\ln b)+J_{\left( \ln b\right) ^{-}}^{\alpha }\left(
f\circ \exp \right) (\ln a)\right\} \leq \frac{f(a)+f(b)}{2}  \label{2-0}
\end{equation}%
with $\alpha >0$.
\end{theorem}

\begin{proof}
Since $f$ is a GA-convex function on $[a,b]$, we have for all $x,y\in
\lbrack a,b]$ (with $t=1/2$ in the inequality (\ref{1-0}))%
\begin{equation*}
f\left( \sqrt{xy}\right) \leq \frac{f(x)+f(y)}{2}.
\end{equation*}%
Choosing $x=a^{t}b^{1-t}$, $y=b^{t}a^{1-t}$, we get%
\begin{equation}
f\left( \sqrt{ab}\right) \leq \frac{f(a^{t}b^{1-t})+f(b^{t}a^{1-t})}{2}
\label{2-0a}
\end{equation}%
Multiplying both sides of (\ref{2-0a}) by $t^{\alpha -1}$, then integrating
the resulting inequality with respect to $t$ over $[0,1]$, we obtain%
\begin{eqnarray*}
f\left( \sqrt{ab}\right) &\leq &\frac{\alpha }{2}\left\{
\dint\limits_{0}^{1}f(a^{t}b^{1-t})dt+\dint\limits_{0}^{1}f(b^{t}a^{1-t})dt%
\right\} \\
&=&\frac{\alpha }{2}\left\{ \dint\limits_{a}^{b}\left( \frac{\ln b-\ln u}{%
\ln b-\ln a}\right) ^{\alpha -1}f(u)\frac{du}{u\ln \frac{b}{a}}%
+\dint\limits_{0}^{1}\left( \frac{\ln u-\ln a}{\ln b-\ln a}\right) ^{\alpha
-1}f(u)\frac{du}{u\ln \frac{b}{a}}\right\} \\
&=&\frac{\alpha }{2\left( \ln \frac{b}{a}\right) ^{\alpha }}\left\{
\dint\limits_{\ln a}^{\ln b}\left( \ln b-t\right) ^{\alpha
-1}f(e^{t})dt+\dint\limits_{\ln a}^{\ln b}\left( t-\ln a\right) ^{\alpha
-1}f(e^{t})dt\right\} \\
&=&\frac{\Gamma (\alpha +1)}{2\left( \ln \frac{b}{a}\right) ^{\alpha }}%
\left\{ J_{\left( \ln a\right) ^{+}}^{\alpha }\left( f\circ \exp \right)
(\ln b)+J_{\left( \ln b\right) ^{-}}^{\alpha }\left( f\circ \exp \right)
(\ln a)\right\}
\end{eqnarray*}%
and the first inequality is proved.

For the proof of the second inequality in (\ref{2-0}) we first note that if $%
f$ is a convex function, then, for $t\in \left[ 0,1\right] $, it yields%
\begin{equation*}
f(a^{t}b^{1-t})\leq tf(a)+(1-t)f(b)
\end{equation*}%
and%
\begin{equation*}
f(b^{t}a^{1-t})\leq tf(b)+(1-t)f(a).
\end{equation*}%
By adding these inequalities we have%
\begin{equation}
f(a^{t}b^{1-t})+f(b^{t}a^{1-t})\leq f(a)+f(b).  \label{2-0b}
\end{equation}%
Then multiplying both sides of (\ref{2-0b}) by $t^{\alpha -1}$, and
integrating the resulting inequality with respect to $t$ over $\left[ 0,1%
\right] $, we obtain%
\begin{equation*}
\dint\limits_{0}^{1}f(a^{t}b^{1-t})t^{\alpha
-1}dt+\dint\limits_{0}^{1}f(b^{t}a^{1-t})t^{\alpha -1}dt\leq \left[ f(a)+f(b)%
\right] \dint\limits_{0}^{1}t^{\alpha -1}dt
\end{equation*}%
i.e.%
\begin{equation*}
\frac{\Gamma (\alpha +1)}{\left( \ln \frac{b}{a}\right) ^{\alpha }}\left\{
J_{\left( \ln a\right) ^{+}}^{\alpha }\left( f\circ \exp \right) (\ln
b)+J_{\left( \ln b\right) ^{-}}^{\alpha }\left( f\circ \exp \right) (\ln
a)\right\} \leq f(a)+f(b).
\end{equation*}%
The proof is completed.
\end{proof}

\begin{lemma}
\label{2.1}Let $f:I\subseteq \left( 0,\infty \right) \rightarrow 
\mathbb{R}
$ be a differentiable function on $I^{\circ }$ such that $f^{\prime }\in
L[a,b]$, where $a,b\in I$ with $a<b$. Then for all $x\in \lbrack a,b]$ , $%
\lambda \in \left[ 0,1\right] $ and $\alpha >0$ we have:%
\begin{eqnarray}
&&I_{f}\left( x,\lambda ,\alpha ,a,b\right) =a\left( \ln \frac{x}{a}\right)
^{\alpha +1}\dint\limits_{0}^{1}\left( t^{\alpha }-\lambda \right) \left( 
\frac{x}{a}\right) ^{t}f^{\prime }\left( x^{t}a^{1-t}\right) dt  \label{2-1}
\\
&&-b\left( \ln \frac{b}{x}\right) ^{\alpha +1}\dint\limits_{0}^{1}\left(
t^{\alpha }-\lambda \right) \left( \frac{x}{b}\right) ^{t}f^{\prime }\left(
x^{t}b^{1-t}\right) dt.  \notag
\end{eqnarray}
\end{lemma}

\begin{proof}
By integration by parts and twice changing the variable, for $x\neq a$ we
can state%
\begin{eqnarray}
&&a\ln \frac{x}{a}\dint\limits_{0}^{1}\left( t^{\alpha }-\lambda \right)
\left( \frac{x}{a}\right) ^{t}f^{\prime }\left( x^{t}a^{1-t}\right) dt
\label{2-1a} \\
&=&\dint\limits_{0}^{1}\left( t^{\alpha }-\lambda \right) df\left(
x^{t}a^{1-t}\right)  \notag \\
&=&\left. \left( t^{\alpha }-\lambda \right) f\left( x^{t}a^{1-t}\right)
\right\vert _{0}^{1}-\frac{\alpha }{\left( \ln \frac{x}{a}\right) ^{\alpha }}%
\dint\limits_{a}^{x}\left( \ln u-\ln a\right) ^{\alpha -1}\frac{f\left(
u\right) }{u}du  \notag \\
&=&\left( 1-\lambda \right) f(x)+\lambda f(a)-\frac{\alpha }{\left( \ln 
\frac{x}{a}\right) ^{\alpha }}\dint\limits_{\ln a}^{\ln x}\left( t-\ln
a\right) ^{\alpha -1}f(e^{t})dt  \notag \\
&=&\left( 1-\lambda \right) f(x)+\lambda f(a)-\frac{\Gamma \left( \alpha
+1\right) }{\left( \ln \frac{x}{a}\right) ^{\alpha }}J_{\left( \ln x\right)
^{-}}^{\alpha }\left( f\circ \exp \right) (\ln a)  \notag
\end{eqnarray}%
and for $x\neq b$ similarly we get%
\begin{eqnarray}
&&-b\ln \frac{b}{x}\dint\limits_{0}^{1}\left( t^{\alpha }-\lambda \right)
\left( \frac{x}{b}\right) ^{t}f^{\prime }\left( x^{t}b^{1-t}\right) dt
\label{2-1b} \\
&=&\dint\limits_{0}^{1}\left( t^{\alpha }-\lambda \right) df\left(
x^{t}b^{1-t}\right)  \notag \\
&=&\left. \left( t^{\alpha }-\lambda \right) f\left( x^{t}b^{1-t}\right)
\right\vert _{0}^{1}-\frac{\alpha }{\left( \ln \frac{b}{x}\right) ^{\alpha }}%
\dint\limits_{x}^{b}\left( \ln b-\ln u\right) ^{\alpha -1}\frac{f\left(
u\right) }{u}du  \notag \\
&=&\left( 1-\lambda \right) f(x)+\lambda f(b)-\frac{\alpha }{\left( \ln 
\frac{b}{x}\right) ^{\alpha }}\dint\limits_{\ln x}^{\ln b}\left( \ln
b-t\right) ^{\alpha -1}f(e^{t})dt  \notag \\
&=&\left( 1-\lambda \right) f(x)+\lambda f(b)-\frac{\Gamma \left( \alpha
+1\right) }{\left( \ln \frac{b}{x}\right) ^{\alpha }}J_{\left( \ln x\right)
^{+}}^{\alpha }\left( f\circ \exp \right) (\ln b)  \notag
\end{eqnarray}%
Multiplying both sides of (\ref{2-1a}) and (\ref{2-1b}) by $\left( \ln \frac{%
x}{a}\right) ^{\alpha }$ and $\left( \ln \frac{b}{x}\right) ^{\alpha }$,
respectively, and adding the resulting identities we obtain the desired
result. For $x=a$ and $x=b$ the identities%
\begin{equation*}
I_{f}\left( a,\lambda ,\alpha ;a,b\right) =b\left( \ln \frac{b}{a}\right)
^{\alpha +1}\dint\limits_{0}^{1}\left( t^{\alpha }-\lambda \right) \left( 
\frac{a}{b}\right) ^{t}f^{\prime }\left( a^{t}b^{1-t}\right) dt,
\end{equation*}%
and%
\begin{equation*}
I_{f}\left( b,\lambda ,\alpha ;a,b\right) =a\left( \ln \frac{b}{a}\right)
^{\alpha +1}\dint\limits_{0}^{1}\left( t^{\alpha }-\lambda \right) \left( 
\frac{b}{a}\right) ^{t}f^{\prime }\left( b^{t}a^{1-t}\right) ,
\end{equation*}%
can be proved respectively easily by performing an integration by parts in
the integrals from the right side and changing the variable.
\end{proof}

\subsection{For GA-$s$-convex functions.}

\begin{theorem}
\label{2.1.1}Let $f:$ $I\subset \left( 0,\infty \right) \rightarrow 
\mathbb{R}
$ be a differentiable function on $I^{\circ }$ such that $f^{\prime }\in
L[a,b]$, where $a,b\in I^{\circ }$ with $a<b$. If $|f^{\prime }|^{q}$ is GA-$%
s$-convex on $[a,b]$ in the second sense for some fixed $q\geq 1$, $x\in
\lbrack a,b]$, $\lambda \in \left[ 0,1\right] $ and $\alpha >0$ then the
following inequality for fractional integrals holds%
\begin{eqnarray}
&&\left\vert I_{f}\left( x,\lambda ,\alpha ,a,b\right) \right\vert  \notag \\
&\leq &A_{1}^{1-\frac{1}{q}}\left( \alpha ,\lambda \right) \left\{ a\left(
\ln \frac{x}{a}\right) ^{\alpha +1}\left( \left\vert f^{\prime }\left(
x\right) \right\vert ^{q}A_{2}\left( x,\alpha ,\lambda ,s,q\right)
+\left\vert f^{\prime }\left( a\right) \right\vert ^{q}A_{3}\left( x,\alpha
,\lambda ,s,q\right) \right) ^{\frac{1}{q}}\right.  \label{2-2} \\
&&+\left. b\left( \ln \frac{b}{x}\right) ^{\alpha +1}\left( \left\vert
f^{\prime }\left( x\right) \right\vert ^{q}A_{4}\left( x,\alpha ,\lambda
,s,q\right) +\left\vert f^{\prime }\left( b\right) \right\vert
^{q}A_{5}\left( x,\alpha ,\lambda ,s,q\right) \right) ^{\frac{1}{q}}\right\}
,  \notag
\end{eqnarray}%
where 
\begin{eqnarray*}
A_{1}\left( \alpha ,\lambda \right) &=&\frac{2\alpha \lambda ^{1+\frac{1}{%
\alpha }}+1}{\alpha +1}-\lambda , \\
A_{2}\left( x,\alpha ,\lambda ,s,q\right) &=&\dint\limits_{0}^{1}\left\vert
t^{\alpha }-\lambda \right\vert \left( \frac{x}{a}\right) ^{qt}t^{s}dt, \\
A_{3}\left( x,\alpha ,\lambda ,s,q\right) &=&\dint\limits_{0}^{1}\left\vert
t^{\alpha }-\lambda \right\vert \left( \frac{x}{a}\right) ^{qt}\left(
1-t\right) ^{s}dt, \\
A_{4}\left( x,\alpha ,\lambda ,s,q\right) &=&\dint\limits_{0}^{1}\left\vert
t^{\alpha }-\lambda \right\vert \left( \frac{x}{b}\right) ^{qt}t^{s}dt, \\
A_{5}\left( x,\alpha ,\lambda ,s,q\right) &=&\dint\limits_{0}^{1}\left\vert
t^{\alpha }-\lambda \right\vert \left( \frac{x}{b}\right) ^{qt}\left(
1-t\right) ^{s}dt.
\end{eqnarray*}
\end{theorem}

\begin{proof}
From Lemma \ref{2.1}, property of the modulus and using the power-mean
inequality we have%
\begin{eqnarray}
&&\left\vert I_{f}\left( x,\lambda ,\alpha ,a,b\right) \right\vert \leq
a\left( \ln \frac{x}{a}\right) ^{\alpha +1}\dint\limits_{0}^{1}\left\vert
t^{\alpha }-\lambda \right\vert \left( \frac{x}{a}\right) ^{t}\left\vert
f^{\prime }\left( x^{t}a^{1-t}\right) \right\vert dt  \notag \\
&&+b\left( \ln \frac{b}{x}\right) ^{\alpha +1}\dint\limits_{0}^{1}\left\vert
t^{\alpha }-\lambda \right\vert \left( \frac{x}{b}\right) ^{t}\left\vert
f^{\prime }\left( x^{t}b^{1-t}\right) \right\vert dt  \notag \\
&\leq &a\left( \ln \frac{x}{a}\right) ^{\alpha +1}\left(
\dint\limits_{0}^{1}\left\vert t^{\alpha }-\lambda \right\vert dt\right) ^{1-%
\frac{1}{q}}\left( \dint\limits_{0}^{1}\left\vert t^{\alpha }-\lambda
\right\vert \left( \frac{x}{a}\right) ^{qt}\left\vert f^{\prime }\left(
x^{t}a^{1-t}\right) \right\vert ^{q}dt\right) ^{\frac{1}{q}}  \notag \\
&&+b\left( \ln \frac{b}{x}\right) ^{\alpha +1}\left(
\dint\limits_{0}^{1}\left\vert t^{\alpha }-\lambda \right\vert dt\right) ^{1-%
\frac{1}{q}}\left( \dint\limits_{0}^{1}\left\vert t^{\alpha }-\lambda
\right\vert \left( \frac{x}{b}\right) ^{qt}\left\vert f^{\prime }\left(
x^{t}b^{1-t}\right) \right\vert ^{q}dt\right) ^{\frac{1}{q}}  \label{2-2a}
\end{eqnarray}%
Since$\left\vert f^{\prime }\right\vert ^{q}$ is GA-$s$-convex on $[a,b]$ we
get%
\begin{eqnarray}
\dint\limits_{0}^{1}\left\vert t^{\alpha }-\lambda \right\vert \left( \frac{x%
}{a}\right) ^{qt}\left\vert f^{\prime }\left( x^{t}a^{1-t}\right)
\right\vert ^{q}dt &\leq &\dint\limits_{0}^{1}\left\vert t^{\alpha }-\lambda
\right\vert \left( \frac{x}{a}\right) ^{qt}\left( t^{s}\left\vert f^{\prime
}\left( x\right) \right\vert ^{q}+\left( 1-t\right) ^{s}\left\vert f^{\prime
}\left( a\right) \right\vert ^{q}\right) dt  \notag \\
&=&\left\vert f^{\prime }\left( x\right) \right\vert ^{q}A_{2}\left(
x,\alpha ,\lambda ,s,q\right) +\left\vert f^{\prime }\left( a\right)
\right\vert ^{q}A_{3}\left( x,\alpha ,\lambda ,s,q\right) ,  \label{2-2b}
\end{eqnarray}%
\begin{eqnarray}
\dint\limits_{0}^{1}\left\vert t^{\alpha }-\lambda \right\vert \left( \frac{x%
}{b}\right) ^{qt}\left\vert f^{\prime }\left( x^{t}b^{1-t}\right)
\right\vert ^{q}dt &\leq &\dint\limits_{0}^{1}\left\vert t^{\alpha }-\lambda
\right\vert \left( \frac{x}{b}\right) ^{qt}\left( t^{s}\left\vert f^{\prime
}\left( x\right) \right\vert ^{q}+\left( 1-t\right) ^{s}\left\vert f^{\prime
}\left( b\right) \right\vert ^{q}\right) dt  \notag \\
&=&\left\vert f^{\prime }\left( x\right) \right\vert ^{q}A_{4}\left(
x,\alpha ,\lambda ,s,q\right) +\left\vert f^{\prime }\left( b\right)
\right\vert ^{q}A_{5}\left( x,\alpha ,\lambda ,s,q\right) ,  \label{2-2c}
\end{eqnarray}%
by simple computation%
\begin{eqnarray}
\dint\limits_{0}^{1}\left\vert t^{\alpha }-\lambda \right\vert dt
&=&\dint\limits_{0}^{\lambda ^{\frac{1}{\alpha }}}\left( \lambda -t^{\alpha
}\right) dt+\dint\limits_{\lambda ^{\frac{1}{\alpha }}}^{1}\left( t^{\alpha
}-\lambda \right) dt  \notag \\
&=&\frac{2\alpha \lambda ^{1+\frac{1}{\alpha }}+1}{\alpha +1}-\lambda .
\label{2-2d}
\end{eqnarray}%
Hence, If we use (\ref{2-2b}), (\ref{2-2c}) and (\ref{2-2d}) in (\ref{2-2a}%
), we obtain the desired result. This completes the proof.
\end{proof}

\begin{corollary}
Under the assumptions of Theorem \ref{2.1.1} with $s=1,$ the inequality (\ref%
{2-2}) reduced to the following inequality%
\begin{eqnarray*}
&&\left\vert I_{f}\left( x,\lambda ,\alpha ,a,b\right) \right\vert \\
&\leq &A_{1}^{1-\frac{1}{q}}\left( \alpha ,\lambda \right) \left\{ a\left(
\ln \frac{x}{a}\right) ^{\alpha +1}\left( \left\vert f^{\prime }\left(
x\right) \right\vert ^{q}A_{2}\left( x,\alpha ,\lambda ,1,q\right)
+\left\vert f^{\prime }\left( a\right) \right\vert ^{q}A_{3}\left( x,\alpha
,\lambda ,1,q\right) \right) ^{\frac{1}{q}}\right. \\
&&+\left. b\left( \ln \frac{b}{x}\right) ^{\alpha +1}\left( \left\vert
f^{\prime }\left( x\right) \right\vert ^{q}A_{4}\left( x,\alpha ,\lambda
,1,q\right) +\left\vert f^{\prime }\left( b\right) \right\vert
^{q}A_{5}\left( x,\alpha ,\lambda ,1,q\right) \right) ^{\frac{1}{q}}\right\}
.
\end{eqnarray*}
\end{corollary}

\begin{corollary}
Under the assumptions of Theorem \ref{2.1.1} with $s=1$ and $\alpha =1,$ the
inequality (\ref{2-2}) reduced to the following inequality%
\begin{eqnarray*}
&&\left\vert I_{f}\left( x,\lambda ,\alpha ,a,b\right) \right\vert \\
&=&\ln \frac{b}{a}\left( 1-\lambda \right) f(x)+\lambda \left[ f(a)\ln \frac{%
x}{a}+f(b)\ln \frac{b}{x}\right] -\alpha \int_{a}^{b}\frac{f(u)}{u}du \\
&\leq &A_{1}^{1-\frac{1}{q}}\left( 1,\lambda \right) \left\{ a\left( \ln 
\frac{x}{a}\right) ^{2}\left( \left\vert f^{\prime }\left( x\right)
\right\vert ^{q}A_{2}\left( x,1,\lambda ,1,q\right) +\left\vert f^{\prime
}\left( a\right) \right\vert ^{q}A_{3}\left( x,1,\lambda ,1,q\right) \right)
^{\frac{1}{q}}\right. \\
&&+\left. b\left( \ln \frac{b}{x}\right) ^{2}\left( \left\vert f^{\prime
}\left( x\right) \right\vert ^{q}A_{4}\left( x,1,\lambda ,1,q\right)
+\left\vert f^{\prime }\left( b\right) \right\vert ^{q}A_{5}\left(
x,1,\lambda ,1,q\right) \right) ^{\frac{1}{q}}\right\} .
\end{eqnarray*}
\end{corollary}

\begin{corollary}
Under the assumptions of Theorem \ref{2.1.1} with $q=1,$ the inequality (\ref%
{2-2}) reduced to the following inequality%
\begin{eqnarray*}
\left\vert I_{f}\left( x,\lambda ,\alpha ,a,b\right) \right\vert &\leq
&\left\{ a\left( \ln \frac{x}{a}\right) ^{\alpha +1}\left( \left\vert
f^{\prime }\left( x\right) \right\vert A_{2}\left( x,\alpha ,\lambda
,s,1\right) +\left\vert f^{\prime }\left( a\right) \right\vert A_{3}\left(
x,\alpha ,\lambda ,s,1\right) \right) \right. \\
&&+\left. b\left( \ln \frac{b}{x}\right) ^{\alpha +1}\left( \left\vert
f^{\prime }\left( x\right) \right\vert A_{4}\left( x,\alpha ,\lambda
,s,1\right) +\left\vert f^{\prime }\left( b\right) \right\vert A_{5}\left(
x,\alpha ,\lambda ,s,1\right) \right) \right\} .
\end{eqnarray*}
\end{corollary}

\begin{corollary}
Under the assumptions of Theorem \ref{2.1.1} with $x=\sqrt{ab},\ \lambda =%
\frac{1}{3},$from the inequality (\ref{2-2}) we get the following Simpson
type inequality for fractional integrals%
\begin{eqnarray*}
&&\left\vert 2^{\alpha -1}\left( \ln \frac{b}{a}\right) ^{-\alpha
}I_{f}\left( \sqrt{ab},\frac{1}{3},\alpha ,a,b\right) \right\vert \\
&=&\left\vert \frac{1}{6}\left[ f(a)+4f\left( \sqrt{ab}\right) +f(b)\right] -%
\frac{2^{\alpha -1}\Gamma \left( \alpha +1\right) }{\left( \ln \frac{b}{a}%
\right) ^{\alpha }}\left[ J_{\left( \ln \sqrt{ab}\right) ^{-}}^{\alpha
}\left( f\circ \exp \right) (\ln a)+J_{\left( \ln \sqrt{ab}\right)
^{+}}^{\alpha }\left( f\circ \exp \right) (\ln b)\right] \right\vert \\
&\leq &\frac{\ln \frac{b}{a}}{4}A_{1}^{1-\frac{1}{q}}\left( \alpha ,\frac{1}{%
3}\right) \left\{ a\left( \left\vert f^{\prime }\left( \sqrt{ab}\right)
\right\vert ^{q}A_{2}\left( \sqrt{ab},\alpha ,\frac{1}{3},s,q\right)
+\left\vert f^{\prime }\left( a\right) \right\vert ^{q}A_{3}\left( \sqrt{ab}%
,\alpha ,\frac{1}{3},s,q\right) \right) ^{\frac{1}{q}}\right. \\
&&\left. +b\left( \left\vert f^{\prime }\left( \sqrt{ab}\right) \right\vert
^{q}A_{4}\left( \sqrt{ab},\alpha ,\frac{1}{3},s,q\right) +\left\vert
f^{\prime }\left( b\right) \right\vert ^{q}A_{5}\left( \sqrt{ab},\alpha ,%
\frac{1}{3},s,q\right) \right) ^{\frac{1}{q}}\right\} .
\end{eqnarray*}
\end{corollary}

\begin{corollary}
Under the assumptions of Theorem \ref{2.1.1} with $x=\sqrt{ab},\ \lambda =0,$%
from the inequality (\ref{2-2}) we get 
\begin{eqnarray*}
&&\left\vert 2^{\alpha -1}\left( \ln \frac{b}{a}\right) ^{-\alpha
}I_{f}\left( \sqrt{ab},0,\alpha ,a,b\right) \right\vert \\
&=&\left\vert f\left( \sqrt{ab}\right) -\frac{2^{\alpha -1}\Gamma \left(
\alpha +1\right) }{\left( \ln \frac{b}{a}\right) ^{\alpha }}\left[ J_{\left(
\ln \sqrt{ab}\right) ^{-}}^{\alpha }\left( f\circ \exp \right) (\ln
a)+J_{\left( \ln \sqrt{ab}\right) ^{+}}^{\alpha }\left( f\circ \exp \right)
(\ln b)\right] \right\vert \\
&\leq &\frac{\ln \frac{b}{a}}{4}\left( \frac{1}{\alpha +1}\right) ^{1-\frac{1%
}{q}}\left\{ a\left( \left\vert f^{\prime }\left( \sqrt{ab}\right)
\right\vert ^{q}A_{2}\left( \sqrt{ab},\alpha ,0,s,q\right) +\left\vert
f^{\prime }\left( a\right) \right\vert ^{q}A_{3}\left( \sqrt{ab},\alpha
,0,s,q\right) \right) ^{\frac{1}{q}}\right. \\
&&\left. +b\left( \left\vert f^{\prime }\left( \sqrt{ab}\right) \right\vert
^{q}A_{4}\left( \sqrt{ab},\alpha ,0,s,q\right) +\left\vert f^{\prime }\left(
b\right) \right\vert ^{q}A_{5}\left( \sqrt{ab},\alpha ,0,s,q\right) \right)
^{\frac{1}{q}}\right\} .
\end{eqnarray*}
\end{corollary}

\begin{corollary}
\label{2.1.1a}Under the assumptions of Theorem \ref{2.1.1} with$\ x=\sqrt{ab}
$ and $\lambda =1,$ from the inequality (\ref{2-2}) we get 
\begin{eqnarray*}
&&\left\vert 2^{\alpha -1}\left( \ln \frac{b}{a}\right) ^{-\alpha
}I_{f}\left( \sqrt{ab},1,\alpha ,a,b\right) \right\vert \\
&=&\left\vert \frac{f(a)+f(b)}{2}-\frac{2^{\alpha -1}\Gamma \left( \alpha
+1\right) }{\left( \ln \frac{b}{a}\right) ^{\alpha }}\left[ J_{\left( \ln 
\sqrt{ab}\right) ^{-}}^{\alpha }\left( f\circ \exp \right) (\ln a)+J_{\left(
\ln \sqrt{ab}\right) ^{+}}^{\alpha }\left( f\circ \exp \right) (\ln b)\right]
\right\vert \\
&\leq &\frac{\ln \frac{b}{a}}{4}\left( \frac{\alpha }{\alpha +1}\right) ^{1-%
\frac{1}{q}}\left\{ a\left[ \left\vert f^{\prime }\left( \sqrt{ab}\right)
\right\vert ^{q}A_{2}\left( \sqrt{ab},\alpha ,1,s,q\right) +\left\vert
f^{\prime }\left( a\right) \right\vert ^{q}A_{3}\left( \sqrt{ab},\alpha
,1,s,q\right) \right] ^{\frac{1}{q}}\right. \\
&&\left. +b\left[ \left\vert f^{\prime }\left( \sqrt{ab}\right) \right\vert
^{q}A_{4}\left( \sqrt{ab},\alpha ,1,s,q\right) +\left\vert f^{\prime }\left(
b\right) \right\vert ^{q}A_{5}\left( \sqrt{ab},\alpha ,1,s,q\right) \right]
^{\frac{1}{q}}\right\} .
\end{eqnarray*}
\end{corollary}

\begin{corollary}
\label{2.1.1b}Let the assumptions of Theorem \ref{2.1.1} hold. If $\
\left\vert f^{\prime }(x)\right\vert \leq M$ for all $x\in \left[ a,b\right] 
$ and $\lambda =0,$ then from the inequality (\ref{2-2}) we get the
following Ostrowski type inequality for fractional integrals%
\begin{eqnarray*}
&&\left\vert \left[ \left( \ln \frac{x}{a}\right) ^{\alpha }+\left( \ln 
\frac{b}{x}\right) ^{\alpha }\right] f(x)-\Gamma \left( \alpha +1\right) %
\left[ J_{\left( \ln x\right) ^{-}}^{\alpha }\left( f\circ \exp \right) (\ln
a)+J_{\left( \ln x\right) ^{+}}^{\alpha }\left( f\circ \exp \right) (\ln b)%
\right] \right\vert \\
&\leq &M\left( \frac{1}{\alpha +1}\right) ^{1-\frac{1}{q}}\left\{ a\left(
\ln \frac{x}{a}\right) ^{\alpha }\left[ A_{2}\left( x,\alpha ,0,s,q\right)
+A_{3}\left( x,\alpha ,0,s,q\right) \right] ^{\frac{1}{q}}\right. \\
&&\left. +b\left( \ln \frac{b}{x}\right) ^{\alpha }\left[ A_{4}\left(
x,\alpha ,0,s,q\right) +A_{5}\left( x,\alpha ,0,s,q\right) \right] ^{\frac{1%
}{q}}\right\}
\end{eqnarray*}%
for each $x\in \left[ a,b\right] .$
\end{corollary}

\subsection{For quasi-geometrically convex functions}

\begin{definition}
A function $f:I\subseteq \left( 0,\infty \right) \mathbb{\rightarrow R}$ is
said to be quasi-geometrically convex on $I$ if 
\begin{equation*}
f\left( x^{t}y^{1-t}\right) \leq \sup \left\{ f(x),f(y)\right\} ,
\end{equation*}%
for any $x,y\in I$ and $t\in \left[ 0,1\right] .$
\end{definition}

Clearly, any GA-convex and geometrically convex functions are
quasi-geometrically convex functions. Furthermore, there exist quasi-convex
functions which are neither GA-convex nor geometrically convex. In that
context, we point out an elementary example. The function $f:\left( 0,4%
\right] \rightarrow 
\mathbb{R}
,$%
\begin{equation*}
f(x)=\left\{ 
\begin{array}{cc}
1, & x\in \left( 0,1\right] \\ 
(x-2)^{2}, & x\in \left[ 1,4\right]%
\end{array}%
\right.
\end{equation*}%
is neither GA-convex nor geometrically convex on $\left( 0,4\right] $ but it
is a quasi-geometrically convex function on $\left( 0,4\right] .$

\begin{theorem}
\label{2.2}Let $f:$ $I\subset \left( 0,\infty \right) \rightarrow 
\mathbb{R}
$ be a differentiable function on $I^{\circ }$ such that $f^{\prime }\in
L[a,b]$, where $a,b\in I^{\circ }$ with $a<b$. If $|f^{\prime }|^{q}$ is
quasi-geometrically convex on $[a,b]$ for some fixed $q\geq 1$, $x\in
\lbrack a,b]$, $\lambda \in \left[ 0,1\right] $ and $\alpha >0$ then the
following inequality for fractional integrals holds%
\begin{eqnarray}
&&\left\vert I_{f}\left( x,\lambda ,\alpha ,a,b\right) \right\vert
\label{2-3} \\
&\leq &A_{1}^{1-\frac{1}{q}}\left( \alpha ,\lambda \right) \left\{ a\left(
\ln \frac{x}{a}\right) ^{\alpha +1}\left( \sup \left\{ \left\vert f^{\prime
}\left( x\right) \right\vert ^{q},\left\vert f^{\prime }\left( a\right)
\right\vert ^{q}\right\} \right) ^{\frac{1}{q}}B_{1}^{\frac{1}{q}}\left(
x,\alpha ,\lambda ,s\right) \right.  \notag \\
&&\left. +b\left( \ln \frac{b}{x}\right) ^{\alpha +1}\left( \sup \left\{
\left\vert f^{\prime }\left( x\right) \right\vert ^{q},\left\vert f^{\prime
}\left( b\right) \right\vert ^{q}\right\} \right) ^{\frac{1}{q}}B_{2}^{\frac{%
1}{q}}\left( x,\alpha ,\lambda ,s\right) \right\}  \notag
\end{eqnarray}%
where 
\begin{eqnarray*}
A_{1}\left( \alpha ,\lambda \right) &=&\frac{2\alpha \lambda ^{1+\frac{1}{%
\alpha }}+1}{\alpha +1}-\lambda , \\
B_{1}\left( x,\alpha ,\lambda ,q\right) &=&\dint\limits_{0}^{1}\left\vert
t^{\alpha }-\lambda \right\vert \left( \frac{x}{a}\right) ^{qt}dt \\
B_{2}\left( x,\alpha ,\lambda ,q\right) &=&\dint\limits_{0}^{1}\left\vert
t^{\alpha }-\lambda \right\vert \left( \frac{x}{b}\right) ^{qt}dt
\end{eqnarray*}
\end{theorem}

\begin{proof}
We proceed similarly as in the proof Theorem \ref{2.1.1}. Since $\left\vert
f^{\prime }\right\vert ^{q}$ is quasi-geometrically convex on $[a,b],$ for
all $t\in \left[ 0,1\right] $%
\begin{equation*}
\left\vert f^{\prime }\left( x^{t}a^{1-t}\right) \right\vert ^{q}\leq \sup
\left\{ \left\vert f^{\prime }\left( x\right) \right\vert ^{q},\left\vert
f^{\prime }\left( a\right) \right\vert ^{q}\right\}
\end{equation*}%
and%
\begin{equation*}
\left\vert f^{\prime }\left( x^{t}b^{1-t}\right) \right\vert ^{q}\leq \sup
\left\{ \left\vert f^{\prime }\left( x\right) \right\vert ^{q},\left\vert
f^{\prime }\left( b\right) \right\vert ^{q}\right\} .
\end{equation*}%
Hence, from the inequality (\ref{2-2a}) we get%
\begin{eqnarray*}
&\leq &a\left( \ln \frac{x}{a}\right) ^{\alpha +1}\left(
\dint\limits_{0}^{1}\left\vert t^{\alpha }-\lambda \right\vert dt\right) ^{1-%
\frac{1}{q}}\left( \dint\limits_{0}^{1}\left\vert t^{\alpha }-\lambda
\right\vert \left( \frac{x}{a}\right) ^{qt}\sup \left\{ \left\vert f^{\prime
}\left( x\right) \right\vert ^{q},\left\vert f^{\prime }\left( a\right)
\right\vert ^{q}\right\} dt\right) ^{\frac{1}{q}} \\
&&+b\left( \ln \frac{b}{x}\right) ^{\alpha +1}\left(
\dint\limits_{0}^{1}\left\vert t^{\alpha }-\lambda \right\vert dt\right) ^{1-%
\frac{1}{q}}\left( \dint\limits_{0}^{1}\left\vert t^{\alpha }-\lambda
\right\vert \left( \frac{x}{b}\right) ^{qt}\sup \left\{ \left\vert f^{\prime
}\left( x\right) \right\vert ^{q},\left\vert f^{\prime }\left( b\right)
\right\vert ^{q}\right\} dt\right) ^{\frac{1}{q}}
\end{eqnarray*}%
\begin{eqnarray*}
&&\left\vert I_{f}\left( x,\lambda ,\alpha ,a,b\right) \right\vert \leq
\left( \dint\limits_{0}^{1}\left\vert t^{\alpha }-\lambda \right\vert
dt\right) ^{1-\frac{1}{q}} \\
&&\times \left\{ a\left( \ln \frac{x}{a}\right) ^{\alpha +1}\left( \sup
\left\{ \left\vert f^{\prime }\left( x\right) \right\vert ^{q},\left\vert
f^{\prime }\left( a\right) \right\vert ^{q}\right\} \right) ^{\frac{1}{q}%
}\left( \dint\limits_{0}^{1}\left\vert t^{\alpha }-\lambda \right\vert
\left( \frac{x}{a}\right) ^{qt}dt\right) ^{\frac{1}{q}}\right. \\
&&\left. +b\left( \ln \frac{b}{x}\right) ^{\alpha +1}\left( \sup \left\{
\left\vert f^{\prime }\left( x\right) \right\vert ^{q},\left\vert f^{\prime
}\left( b\right) \right\vert ^{q}\right\} \right) ^{\frac{1}{q}}\left(
\dint\limits_{0}^{1}\left\vert t^{\alpha }-\lambda \right\vert \left( \frac{x%
}{b}\right) ^{qt}dt\right) ^{\frac{1}{q}}\right\} \\
&\leq &A_{1}^{1-\frac{1}{q}}\left( \alpha ,\lambda \right) \left\{ a\left(
\ln \frac{x}{a}\right) ^{\alpha +1}\left( \sup \left\{ \left\vert f^{\prime
}\left( x\right) \right\vert ^{q},\left\vert f^{\prime }\left( a\right)
\right\vert ^{q}\right\} \right) ^{\frac{1}{q}}B_{1}^{\frac{1}{q}}\left(
x,\alpha ,\lambda ,q\right) \right. \\
&&\left. +b\left( \ln \frac{b}{x}\right) ^{\alpha +1}\left( \sup \left\{
\left\vert f^{\prime }\left( x\right) \right\vert ^{q},\left\vert f^{\prime
}\left( b\right) \right\vert ^{q}\right\} \right) ^{\frac{1}{q}}B_{2}^{\frac{%
1}{q}}\left( x,\alpha ,\lambda ,q\right) \right\}
\end{eqnarray*}%
which completes the proof.
\end{proof}

\begin{corollary}
Under the assumptions of Theorem \ref{2.2} with $q=1,$ the inequality (\ref%
{2-3}) reduced to the following inequality%
\begin{eqnarray*}
\left\vert I_{f}\left( x,\lambda ,\alpha ,a,b\right) \right\vert &\leq
&\left\{ a\left( \ln \frac{x}{a}\right) ^{\alpha +1}B_{1}\left( x,\alpha
,\lambda ,1\right) \sup \left\{ \left\vert f^{\prime }\left( x\right)
\right\vert ,\left\vert f^{\prime }\left( a\right) \right\vert \right\}
\right. \\
&&+\left. b\left( \ln \frac{b}{x}\right) ^{\alpha +1}B_{2}\left( x,\alpha
,\lambda ,1\right) \sup \left\{ \left\vert f^{\prime }\left( x\right)
\right\vert ,\left\vert f^{\prime }\left( b\right) \right\vert \right\}
\right\} .
\end{eqnarray*}
\end{corollary}

\begin{corollary}
Under the assumptions of Theorem \ref{2.2} with $x=\sqrt{ab},\ \lambda =%
\frac{1}{3},$ from the inequality (\ref{2-3}) we get the following Simpson
type inequality or fractional integrals%
\begin{eqnarray*}
&&\left\vert \frac{1}{6}\left[ f(a)+4f\left( \sqrt{ab}\right) +f(b)\right] -%
\frac{2^{\alpha -1}\Gamma \left( \alpha +1\right) }{\left( \ln \frac{b}{a}%
\right) ^{\alpha }}\left[ J_{\left( \ln \sqrt{ab}\right) ^{-}}^{\alpha
}\left( f\circ \exp \right) (\ln a)+J_{\left( \ln \sqrt{ab}\right)
^{+}}^{\alpha }\left( f\circ \exp \right) (\ln b)\right] \right\vert \\
&\leq &\frac{\ln \frac{b}{a}}{4}A_{1}^{1-\frac{1}{q}}\left( \alpha ,\frac{1}{%
3}\right) \left\{ a\left[ \sup \left\{ \left\vert f^{\prime }\left( \sqrt{ab}%
\right) \right\vert ,\left\vert f^{\prime }\left( a\right) \right\vert
\right\} \right] ^{\frac{1}{q}}B_{1}^{\frac{1}{q}}\left( \sqrt{ab},\alpha ,%
\frac{1}{3},q\right) \right. \\
&&\left. +b\left[ \sup \left\{ \left\vert f^{\prime }\left( \sqrt{ab}\right)
\right\vert ,\left\vert f^{\prime }\left( b\right) \right\vert \right\} %
\right] ^{\frac{1}{q}}B_{2}^{\frac{1}{q}}\left( \sqrt{ab},\alpha ,\frac{1}{3}%
,q\right) \right\}
\end{eqnarray*}
\end{corollary}

\begin{corollary}
Under the assumptions of Theorem \ref{2.2} with $x=\sqrt{ab},\ \lambda =0,$%
from the inequality (\ref{2-3}) we get the following midpoint type
inequality or fractional integrals%
\begin{eqnarray*}
&&\left\vert f\left( \sqrt{ab}\right) -\frac{2^{\alpha -1}\Gamma \left(
\alpha +1\right) }{\left( \ln \frac{b}{a}\right) ^{\alpha }}\left[ J_{\left(
\ln \sqrt{ab}\right) ^{-}}^{\alpha }\left( f\circ \exp \right) (\ln
a)+J_{\left( \ln \sqrt{ab}\right) ^{+}}^{\alpha }\left( f\circ \exp \right)
(\ln b)\right] \right\vert \\
&\leq &\frac{\ln \frac{b}{a}}{4}\left( \frac{1}{\alpha +1}\right) ^{1-\frac{1%
}{q}}\left\{ a\left[ \sup \left\{ \left\vert f^{\prime }\left( \sqrt{ab}%
\right) \right\vert ,\left\vert f^{\prime }\left( a\right) \right\vert
\right\} \right] ^{\frac{1}{q}}B_{1}^{\frac{1}{q}}\left( \sqrt{ab},\alpha
,0,q\right) \right. \\
&&\left. +b\left[ \sup \left\{ \left\vert f^{\prime }\left( \sqrt{ab}\right)
\right\vert ,\left\vert f^{\prime }\left( b\right) \right\vert \right\} %
\right] ^{\frac{1}{q}}B_{2}^{\frac{1}{q}}\left( \sqrt{ab},\alpha ,0,q\right)
\right\}
\end{eqnarray*}
\end{corollary}

\begin{corollary}
Under the assumptions of Theorem \ref{2.2} with$\ x=\sqrt{ab}$, $\lambda =1,$%
from the inequality (\ref{2-3}) we get the following trapezoid type
inequality or fractional integrals%
\begin{eqnarray*}
&&\left\vert \frac{f(a)+f(b)}{2}-\frac{2^{\alpha -1}\Gamma \left( \alpha
+1\right) }{\left( \ln \frac{b}{a}\right) ^{\alpha }}\left[ J_{\left( \ln 
\sqrt{ab}\right) ^{-}}^{\alpha }\left( f\circ \exp \right) (\ln a)+J_{\left(
\ln \sqrt{ab}\right) ^{+}}^{\alpha }\left( f\circ \exp \right) (\ln b)\right]
\right\vert \\
&\leq &\frac{\ln \frac{b}{a}}{4}\left( \frac{1}{\alpha +1}\right) ^{1-\frac{1%
}{q}}B_{1}^{\frac{1}{q}}\left( \sqrt{ab},\alpha ,1,q\right) \\
&&\times \left\{ a\left[ \sup \left\{ \left\vert f^{\prime }\left( \sqrt{ab}%
\right) \right\vert ,\left\vert f^{\prime }\left( a\right) \right\vert
\right\} \right] ^{\frac{1}{q}}+b\left[ \sup \left\{ \left\vert f^{\prime
}\left( \sqrt{ab}\right) \right\vert ,\left\vert f^{\prime }\left( b\right)
\right\vert \right\} \right] ^{\frac{1}{q}}\right\}
\end{eqnarray*}
\end{corollary}

\begin{corollary}
Let the assumptions of Theorem \ref{2.2} hold. If $\ \left\vert f^{\prime
}(x)\right\vert \leq M$ for all $x\in \left[ a,b\right] $ and $\lambda =0,$
then from the inequality (\ref{2-3}) we get the following Ostrowski type
inequality or fractional integrals%
\begin{eqnarray*}
&&\left\vert \left[ \left( \ln \frac{x}{a}\right) ^{\alpha }+\left( \ln 
\frac{b}{x}\right) ^{\alpha }\right] f(x)-\Gamma \left( \alpha +1\right) %
\left[ J_{\left( \ln x\right) ^{-}}^{\alpha }\left( f\circ \exp \right) (\ln
a)+J_{\left( \ln x\right) ^{+}}^{\alpha }\left( f\circ \exp \right) (\ln b)%
\right] \right\vert \\
&\leq &\frac{M}{\left( \alpha +1\right) ^{1-\frac{1}{q}}}\left[ a\left( \ln 
\frac{x}{a}\right) ^{\alpha +1}B_{1}^{\frac{1}{q}}\left( x,\alpha
,0,q\right) +b\left( \ln \frac{b}{x}\right) ^{\alpha +1}B_{2}^{\frac{1}{q}%
}\left( x,\alpha ,0,q\right) \right] ,
\end{eqnarray*}
\end{corollary}

\subsection{For $\left( s,m\right) $-convex functions}

Similarly lemma \ref{2.1}, we can proved the folllowing lemma.

\begin{lemma}
\label{2.1a}Let $f:I\subseteq \left( 0,\infty \right) \rightarrow 
\mathbb{R}
$ be a differentiable function on $I^{\circ }$ such that $f^{\prime }\in
L[a^{m},b^{m}]$, where $a^{m},b\in I$ with $a<b$. Then for all $x\in \lbrack
a,b]$, $\lambda \in \left[ 0,1\right] $ and $\alpha >0$ we have:%
\begin{eqnarray*}
&&I_{f}\left( x^{m},\lambda ,\alpha ,a^{m},b^{m}\right) =m^{\alpha
+1}a^{m}\left( \ln \frac{x}{a}\right) ^{\alpha +1}\dint\limits_{0}^{1}\left(
t^{\alpha }-\lambda \right) \left( \frac{x}{a}\right) ^{mt}f^{\prime }\left(
x^{mt}a^{m\left( 1-t\right) }\right) dt \\
&&+m^{\alpha +1}b^{m}\left( \ln \frac{b}{x}\right) ^{\alpha
+1}\dint\limits_{0}^{1}\left( t^{\alpha }-\lambda \right) \left( \frac{x}{b}%
\right) ^{mt}f^{\prime }\left( x^{mt}b^{m\left( 1-t\right) }\right) dt.
\end{eqnarray*}
\end{lemma}

\begin{theorem}
\label{2.3}Let $f:I\subseteq \left( 0,\infty \right) \rightarrow 
\mathbb{R}
$ be a differentiable function on $I^{\circ }$ such that $f^{\prime }\in
L[a^{m},b^{m}]$, where $a^{m},b\in I$ $^{\circ }$ with $a<b$. If $|f^{\prime
}|^{q}$ is $\left( s,m\right) $-convex on $[a^{m},b]$ for some fixed $q\geq
1 $, $x\in \lbrack a,b]$, $\lambda \in \left[ 0,1\right] $ and $\alpha >0$
then the following inequality for fractional integrals holds%
\begin{eqnarray}
&&\left\vert I_{f}\left( x^{m},\lambda ,\alpha ,a^{m},b^{m}\right)
\right\vert \leq m^{\alpha +1}A_{1}\left( \alpha ,\lambda \right) ^{1-\frac{1%
}{q}}  \label{2-4} \\
&&\times \left\{ a^{m}\left( \ln \frac{x}{a}\right) ^{\alpha +1}\left(
\left\vert f^{\prime }\left( x^{m}\right) \right\vert ^{q}C_{1}\left(
x,\alpha ,\lambda ,q,m,s\right) +m\left\vert f^{\prime }\left( a\right)
\right\vert ^{q}C_{2}\left( x,\alpha ,\lambda ,q,m,s\right) \right) ^{\frac{1%
}{q}}\right.  \notag \\
&&\left. +b^{m}\left( \ln \frac{b}{x}\right) ^{\alpha +1}\left( \left\vert
f^{\prime }\left( x^{m}\right) \right\vert ^{q}C_{3}\left( x,\alpha ,\lambda
,q,m,s\right) +m\left\vert f^{\prime }\left( b\right) \right\vert
^{q}C_{4}\left( x,\alpha ,\lambda ,q,m,s\right) \right) ^{\frac{1}{q}%
}\right\}  \notag
\end{eqnarray}%
where 
\begin{eqnarray*}
A_{1}\left( \alpha ,\lambda \right) &=&\frac{2\alpha \lambda ^{1+\frac{1}{%
\alpha }}+1}{\alpha +1}-\lambda , \\
C_{1}\left( x,\alpha ,\lambda ,q,m,s\right)
&=&\dint\limits_{0}^{1}\left\vert t^{\alpha }-\lambda \right\vert \left( 
\frac{x}{a}\right) ^{qmt}t^{s}dt, \\
C_{2}\left( x,\alpha ,\lambda ,q,m,s\right)
&=&\dint\limits_{0}^{1}\left\vert t^{\alpha }-\lambda \right\vert \left( 
\frac{x}{a}\right) ^{qmt}\left( 1-t^{s}\right) dt, \\
C_{3}\left( x,\alpha ,\lambda ,q,m,s\right)
&=&\dint\limits_{0}^{1}\left\vert t^{\alpha }-\lambda \right\vert \left( 
\frac{x}{b}\right) ^{qmt}t^{s}dt, \\
C_{4}\left( x,\alpha ,\lambda ,q,m,s\right)
&=&\dint\limits_{0}^{1}\left\vert t^{\alpha }-\lambda \right\vert \left( 
\frac{x}{b}\right) ^{qmt}\left( 1-t^{s}\right) dt.
\end{eqnarray*}
\end{theorem}

\begin{proof}
We proceed similarly as in the proof Theorem \ref{2.1.1}. From Lemma \ref%
{2.1a}, property of the modulus and using the power-mean inequality we have%
\begin{eqnarray*}
&&I_{f}\left( x^{m},\lambda ,\alpha ,a^{m},b^{m}\right) =m^{\alpha
+1}a^{m}\left( \ln \frac{x}{a}\right) ^{\alpha +1}\dint\limits_{0}^{1}\left(
t^{\alpha }-\lambda \right) \left( \frac{x}{a}\right) ^{mt}f^{\prime }\left(
x^{mt}a^{m\left( 1-t\right) }\right) dt \\
&&+m^{\alpha +1}b^{m}\left( \ln \frac{b}{x}\right) ^{\alpha
+1}\dint\limits_{0}^{1}\left( t^{\alpha }-\lambda \right) \left( \frac{x}{b}%
\right) ^{mt}f^{\prime }\left( x^{mt}b^{m\left( 1-t\right) }\right) dt.
\end{eqnarray*}%
\begin{equation*}
\left\vert I_{f}\left( x^{m},\lambda ,\alpha ,a^{m},b^{m}\right) \right\vert
\leq m^{\alpha +1}\left( \dint\limits_{0}^{1}\left\vert t^{\alpha }-\lambda
\right\vert dt\right) ^{1-\frac{1}{q}}
\end{equation*}%
\begin{eqnarray}
&&\times \left\{ a^{m}\left( \ln \frac{x}{a}\right) ^{\alpha +1}\left(
\dint\limits_{0}^{1}\left\vert t^{\alpha }-\lambda \right\vert \left( \frac{x%
}{a}\right) ^{qmt}\left\vert f^{\prime }\left( x^{mt}a^{m\left( 1-t\right)
}\right) \right\vert ^{q}dt\right) ^{\frac{1}{q}}\right.  \notag \\
&&\left. +b^{m}\left( \ln \frac{b}{x}\right) ^{\alpha +1}\left(
\dint\limits_{0}^{1}\left\vert t^{\alpha }-\lambda \right\vert \left( \frac{x%
}{b}\right) ^{qmt}\left\vert f^{\prime }\left( x^{mt}b^{m\left( 1-t\right)
}\right) \right\vert ^{q}dt\right) ^{\frac{1}{q}}\right\} .  \label{2-4a}
\end{eqnarray}%
Since $|f^{\prime }|^{q}$ is $\left( s,m\right) $-convex on $[a^{m},b],$ for
all $t\in \left[ 0,1\right] $%
\begin{equation}
\left\vert f^{\prime }\left( x^{mt}a^{m\left( 1-t\right) }\right)
\right\vert ^{q}\leq t^{s}\left\vert f^{\prime }\left( x^{m}\right)
\right\vert ^{q}+m\left( 1-t^{s}\right) \left\vert f^{\prime }\left(
a\right) \right\vert ^{q},  \label{2-4b}
\end{equation}%
\begin{equation}
\left\vert f^{\prime }\left( x^{mt}b^{m\left( 1-t\right) }\right)
\right\vert ^{q}\leq t^{s}\left\vert f^{\prime }\left( x^{m}\right)
\right\vert ^{q}+m\left( 1-t^{s}\right) \left\vert f^{\prime }\left(
b\right) \right\vert ^{q}.  \label{2-4c}
\end{equation}%
If we use (\ref{2-4b}), (\ref{2-4c}) and (\ref{2-2d}) in (\ref{2-4a}), we
obtain the desired result. This completes the proof.
\end{proof}

\begin{corollary}
Under the assumptions of Theorem \ref{2.3} with $q=1,$ the inequality (\ref%
{2-4}) reduced to the following inequality%
\begin{eqnarray*}
&&I_{f}\left( x^{m},\lambda ,\alpha ,a^{m},b^{m}\right) \leq m^{\alpha +1} \\
&&\times \left\{ a^{m}\left( \ln \frac{x}{a}\right) ^{\alpha +1}\left(
\left\vert f^{\prime }\left( x^{m}\right) \right\vert C_{1}\left( x,\alpha
,\lambda ,1,m,s\right) +m\left\vert f^{\prime }\left( a\right) \right\vert
C_{2}\left( x,\alpha ,\lambda ,1,m,s\right) \right) \right. \\
&&\left. +b^{m}\left( \ln \frac{b}{x}\right) ^{\alpha +1}\left( \left\vert
f^{\prime }\left( x^{m}\right) \right\vert C_{3}\left( x,\alpha ,\lambda
,1,m,s\right) +m\left\vert f^{\prime }\left( b\right) \right\vert
C_{4}\left( x,\alpha ,\lambda ,1,m,s\right) \right) \right\}
\end{eqnarray*}
\end{corollary}

\begin{corollary}
Under the assumptions of Theorem \ref{2.3} with $x=\sqrt{ab},\ \lambda =%
\frac{1}{3},$from the inequality (\ref{2-4}) we get the following Simpson
type inequality or fractional integrals%
\begin{eqnarray*}
&&2^{\alpha -1}\left( m\ln \frac{b}{a}\right) ^{-\alpha }\left\vert
I_{f}\left( \left( \sqrt{ab}\right) ^{m},\frac{1}{3},\alpha
,a^{m},b^{m}\right) \right\vert \\
&=&\left\vert \frac{1}{6}\left[ f(a^{m})+4f\left( \left( \sqrt{ab}\right)
^{m}\right) +f(b^{m})\right] -\frac{2^{\alpha -1}\Gamma \left( \alpha
+1\right) }{\left( m\ln \frac{b}{a}\right) ^{\alpha }}\right. \\
&&\times \left. \left[ J_{\left( m\ln \sqrt{ab}\right) ^{-}}^{\alpha }\left(
f\circ \exp \right) (m\ln a)+J_{\left( m\ln \sqrt{ab}\right) ^{+}}^{\alpha
}\left( f\circ \exp \right) (m\ln b)\right] \right\vert
\end{eqnarray*}%
\begin{eqnarray*}
&&\leq \frac{m\ln \frac{b}{a}}{4}A_{1}^{1-\frac{1}{q}}\left( \alpha ,\frac{1%
}{3}\right) \\
&&\times \left\{ \left( \left\vert f^{\prime }\left( \left( \sqrt{ab}\right)
^{m}\right) \right\vert ^{q}C_{1}\left( \sqrt{ab},\alpha ,\frac{1}{3}%
,q,m,s\right) +m\left\vert f^{\prime }\left( a\right) \right\vert
^{q}C_{2}\left( \sqrt{ab},\alpha ,\frac{1}{3},q,m,s\right) \right) ^{\frac{1%
}{q}}\right. \\
&&\left. +\left( \left\vert f^{\prime }\left( \left( \sqrt{ab}\right)
^{m}\right) \right\vert ^{q}C_{3}\left( \sqrt{ab},\alpha ,\frac{1}{3}%
,q,m,s\right) +m\left\vert f^{\prime }\left( b\right) \right\vert
^{q}C_{4}\left( \sqrt{ab},\alpha ,\frac{1}{3},q,m,s\right) \right) ^{\frac{1%
}{q}}\right\} .
\end{eqnarray*}
\end{corollary}

\begin{corollary}
Under the assumptions of Theorem \ref{2.3} with $x=\sqrt{ab},\ \lambda =0,$%
from the inequality (\ref{2-4}) we get the following inequality%
\begin{eqnarray*}
&&2^{\alpha -1}\left( m\ln \frac{b}{a}\right) ^{-\alpha }\left\vert
I_{f}\left( \left( \sqrt{ab}\right) ^{m},0,\alpha ,a^{m},b^{m}\right)
\right\vert \\
&=&\left\vert f\left( \left( \sqrt{ab}\right) ^{m}\right) -\frac{2^{\alpha
-1}\Gamma \left( \alpha +1\right) }{\left( m\ln \frac{b}{a}\right) ^{\alpha }%
}\left[ J_{\left( m\ln \sqrt{ab}\right) ^{-}}^{\alpha }\left( f\circ \exp
\right) (m\ln a)+J_{\left( m\ln \sqrt{ab}\right) ^{+}}^{\alpha }\left(
f\circ \exp \right) (m\ln b)\right] \right\vert \\
&\leq &\frac{m\ln \frac{b}{a}}{4}\left( \frac{1}{\alpha +1}\right) ^{1-\frac{%
1}{q}}\left\{ \left[ \left\vert f^{\prime }\left( \left( \sqrt{ab}\right)
^{m}\right) \right\vert ^{q}C_{1}\left( \sqrt{ab},\alpha ,0,q,m,s\right)
+m\left\vert f^{\prime }\left( a\right) \right\vert ^{q}C_{2}\left( \sqrt{ab}%
,\alpha ,0,q,m,s\right) \right] ^{\frac{1}{q}}\right. \\
&&\left. +\left[ \left\vert f^{\prime }\left( \left( \sqrt{ab}\right)
^{m}\right) \right\vert ^{q}C_{3}\left( \sqrt{ab},\alpha ,0,q,m,s\right)
+m\left\vert f^{\prime }\left( b\right) \right\vert ^{q}C_{4}\left( \sqrt{ab}%
,\alpha ,0,q,m,s\right) \right] ^{\frac{1}{q}}\right\} .
\end{eqnarray*}
\end{corollary}

\begin{corollary}
Under the assumptions of Theorem \ref{2.3} with$\ x=\sqrt{ab},\ \lambda =1,$%
from the inequality (\ref{2-4}) we get the following inequality 
\begin{eqnarray*}
&&2^{\alpha -1}\left( m\ln \frac{b}{a}\right) ^{-\alpha }\left\vert
I_{f}\left( \left( \sqrt{ab}\right) ^{m},1,\alpha ,a^{m},b^{m}\right)
\right\vert \\
&=&\left\vert \frac{f(a^{m})+f(b^{m})}{2}-\frac{2^{\alpha -1}\Gamma \left(
\alpha +1\right) }{\left( m\ln \frac{b}{a}\right) ^{\alpha }}\left[
J_{\left( m\ln \sqrt{ab}\right) ^{-}}^{\alpha }\left( f\circ \exp \right)
(m\ln a)+J_{\left( m\ln \sqrt{ab}\right) ^{+}}^{\alpha }\left( f\circ \exp
\right) (m\ln b)\right] \right\vert \\
&\leq &\frac{m\ln \frac{b}{a}}{4}\left( \frac{\alpha }{\alpha +1}\right) ^{1-%
\frac{1}{q}}\left\{ \left[ \left\vert f^{\prime }\left( \left( \sqrt{ab}%
\right) ^{m}\right) \right\vert ^{q}C_{1}\left( \sqrt{ab},\alpha
,1,q,m,s\right) +m\left\vert f^{\prime }\left( a\right) \right\vert
^{q}C_{2}\left( \sqrt{ab},\alpha ,1,q,m,s\right) \right] ^{\frac{1}{q}%
}\right. \\
&&\left. +\left[ \left\vert f^{\prime }\left( \left( \sqrt{ab}\right)
^{m}\right) \right\vert ^{q}C_{3}\left( \sqrt{ab},\alpha ,1,q,m,s\right)
+m\left\vert f^{\prime }\left( b\right) \right\vert ^{q}C_{4}\left( \sqrt{ab}%
,\alpha ,1,q,m,s\right) \right] ^{\frac{1}{q}}\right\} .
\end{eqnarray*}
\end{corollary}

\begin{corollary}
Let the assumptions of Theorem \ref{2.3} hold. If $\ \left\vert f^{\prime
}(u)\right\vert \leq M$ for all $u\in \left[ a^{m},b\right] $ and $\lambda
=0,$ then from the inequality (\ref{2-4}) we get the following Ostrowski
type inequality or fractional integrals%
\begin{eqnarray*}
&&\left\vert \left[ \left( \ln \frac{x}{a}\right) ^{\alpha }+\left( \ln 
\frac{b}{x}\right) ^{\alpha }\right] f(x^{m})-\frac{\Gamma \left( \alpha
+1\right) }{m^{\alpha }}\left[ J_{\left( m\ln x\right) ^{-}}^{\alpha }\left(
f\circ \exp \right) (m\ln a)+J_{\left( m\ln x\right) ^{+}}^{\alpha }\left(
f\circ \exp \right) (m\ln b)\right] \right\vert \\
&\leq &\frac{mM}{\left( \alpha +1\right) ^{1-\frac{1}{q}}}\left\{
a^{m}\left( \ln \frac{x}{a}\right) ^{\alpha +1}\left( C_{1}\left( x,\alpha
,0,q,m,s\right) +mC_{2}\left( x,\alpha ,0,q,m,s\right) \right) ^{\frac{1}{q}%
}\right. \\
&&\left. +b^{m}\left( \ln \frac{b}{x}\right) ^{\alpha +1}\left( C_{3}\left(
x,\alpha ,\lambda ,q,m,s\right) +mC_{4}\left( x,\alpha ,\lambda
,q,m,s\right) \right) ^{\frac{1}{q}}\right\}
\end{eqnarray*}%
for each $x\in \left[ a,b\right] .$
\end{corollary}


\begin{thebibliography}{99}
\bibitem{ADDC10} M. Alomaria, M. Darus, S.S. Dragomir, P. Cerone, Ostrowski
type inequalities for functions whose derivatives are $s$-convex in the
second sense, Applied Mathematics Letters 23 (2010) 1071--1076.

\bibitem{AKO11} M. Avci, H. Kavurmaci and M.E. Ozdemir, New inequalities of
Hermite-Hadamard type via $s$-convex functions in the second sense with
applications, Appl. Math. Comput., 217 (2011) 5171--5176.

\bibitem{D10} Z. Dahmani, On Minkowski and Hermite-Hadamard integral
inequalities via fractional via fractional integration, Ann. Funct. Anal. 1
(1) (2010), 51-58

\bibitem{GM97} R. Gorenflo, F. Mainardi, Fractional calculus; integral and
differential equations of fractional order, Springer Verlag, Wien (1997),
223-276.

\bibitem{I13} I. Iscan, A new generalization of some integral inequalities
for $(\alpha ,m)$-convex functions, Mathematical Sciences\textit{,} (2013),
doi:10.1186/2251-7456-7-22.

\bibitem{I13b} I. Iscan, New estimates on generalization of some integral
inequalities for $(\alpha ,m)$-convex functions, Contemporary Analysis and
Applied Mathematics, accepted.

\bibitem{I13c} I. Iscan, New estimates on generalization of some integral
inequalities for $s$-convex functions and their applications, International
Journal of Pure and Applied Mathematics, 86 (4) (2013) accepted.

\bibitem{JZQ13} A.P. Ji, T.Y. Zhang and F. Qi, Integral inequalities of
Hermite-Hadamard type for $(\alpha ,m)$-GA-convex functions,
arXiv:1306.0852. Available online at http://arxiv.org/abs/1306.0852.

\bibitem{MR93} S. Miller and B. Ross, An introduction to the Fractional
Calculus and Fractional Differential Equations, John Wiley \& Sons, USA
(1993), 2.

\bibitem{N00} C. P. Niculescu, Convexity according to the geometric mean,
Math. Inequal. Appl. 3 (2) (2000), 155-167. Available online at
http://dx.doi.org/10.7153/mia-03-19.

\bibitem{N03} C. P. Niculescu, Convexity according to means, Math. Inequal.
Appl. 6 (4) (2003), 571-579. Available online at
http://dx.doi.org/10.7153/mia-06-53.

\bibitem{OAK12} M.E. Ozdemir, M. Avci, H. Kavurmaci, Hermite-Hadamard type
inequalities for $s$-convex and $s$-concave functions via fractional
integrals, arXiv:1202.0380v1.

\bibitem{P11} J. Park, Generalization of some Simpson-like type inequalities
via differentiable $s$-convex mappings in the second sense, International
journal of Mathematics and Mathematical Sciences, vol. 2011, Article ID
493531, 13 pages, doi:10.1155/493531.

\bibitem{P99} I. Podlubni, Fractional Differential Equations, Academic
Press, San Diego, 1999.

\bibitem{SA11} M.Z. Sar\i kaya and N. Aktan, On the generalization of some
integral inequalities and \ their applications, Mathematical and Computer
Modelling, 54 (2011) 2175- 2182.

\bibitem{S12} E. Set, New inequalities of Ostrowski type for mapping whose
derivatives are $s$-convex in the second sense via fractional integrals,
Computers and Math. with Appl. 63 (2012) 1147-1154.

\bibitem{SO12} M.Z. Sar\i kaya and H. Ogunmez, On new inequalities via
Riemann-Liouville fractional integration, Abstract an Applied Analysis, vol.
2012, Article ID 428983, 10 pages, doi:10.1155/2012/428983.

\bibitem{SOS12} E. Set, M.E. Ozdemir and M.Z. Sar\i kaya, On new
inequalities of Simpson's type for quasi-convex functions with applications,
Tamkang Journal of Mathematics, 43 (3) (2012) 357-364.

\bibitem{SSO10} M.Z. Sar\i kaya, E. Set and M.E. Ozdemir, On new
inequalities of Simpson's type for $s$-convex functions, Computers and Math.
with Appl., 60 (2010) 2191-2199.

\bibitem{SSYB11} M.Z. Sar\i kaya, E. Set, H. Yald\i z and N. Ba\c{s}ak,
Hermite-Hadamard's inequalities for fractional integrals and related
fractional inequalities, Mathematical and Computer Modelling,
DOI:10.1016/j.mcm.2011.12.048.

\bibitem{SYQ13} Y. Shuang, H.-P. Yin, and F. Qi, Hermite-Hadamard type
integral inequalities for geometric-arithmetically $s$-convex functions,
Analysis (Munich) 33 (2) (2013), 197-208. Available online at
http://dx.doi.org/10.1524/anly.2013.1192.

\bibitem{ZJQ12} T.-Y. Zhang, A.-P. Ji and F. Qi, On Integral Inequalities of
Hermite-Hadamard Type for $s$-Geometrically Convex Functions, \textit{%
Abstract and Applied Analysis}, \textbf{2012} (2012), Article ID 560586, 14
pages, doi:10.1155/2012/560586.
\end{thebibliography}
\end{document}